\documentclass[12pt, reqno]{amsproc}
\usepackage{graphicx, enumerate}
\usepackage[margin=1in]{geometry}
\usepackage{amsmath}
\usepackage{amsfonts}
\usepackage{amssymb}
\usepackage{amsfonts,amsmath,amssymb}
\usepackage{amsthm}
\newtheorem{lemma}{Lemma}
\newtheorem{theorem}{Theorem}
\newtheorem{crlly}{Corollary}
\newtheorem{eg}{Example}
\title[Growth of Solutions of Second Order Linear Differential Equations]{Growth of Solutions of Second Order Linear Differential Equations}
\author[S. Kumar, N. Mehra and M. Saini]{Sanjay Kumar, Naveen Mehra and Manisha Saini}
\address{Sanjay Kumar, Department of Mathematics, Deen Dayal Upadhyaya College, University of Delhi, Delhi-110007, India}
\email{skpant@ddu.du.ac.in}
\address{Naveen Mehra, Department of Mathematics, Kumaun University, S.S.J Campus, Almora-263601, Uttarakhand, India}
\email{naveenmehra00@gmail.com}
\address{Manisha Saini, Department of Mathematics, University of Delhi, Delhi-110007, India}
\email{sainimanisha210@gmail.com}
\subjclass[2010]{34M10, 30D35}
\keywords{entire function, meromorphic function, lower order of growth, order of growth, complex differential equation, critical rays, blows up exponentially and decays to zero exponentially.}
\thanks{The research work of the second author is supported by research fellowship from Council of Scientific and Industrial Research (CSIR), New Delhi and of the last author from University Grants Commission (UGC), New Delhi.} 
\begin{document}
\maketitle
\begin{abstract}
In this paper, we will prove that all non-trivial solutions of $f''+A(z)f'+B(z)f=0$ are of infinite order, where we have some restrictions on entire functions $A(z)$ and $B(z)$.
\end{abstract}

\section{Introduction and statement of main results}
Consider the differential equation,
\begin{equation}\label{eq1}
f''+A(z)f'+B(z)f=0
\end{equation}
where $A(z)$ and $B(z)$ are entire functions. It is well known result that solutions of \eqref{eq1} are also entire functions. All solutions of \eqref{eq1} are of finite order if and only if both coefficients are polynomials. If either of $A(z)$ or $B(z)$ is transcendental entire then almost all solutions are of infinite order. Gundersen\cite{gundersen2} proved that any non-constant solution is of infinite order when $\rho(A)<\rho(B)$. Hellerstein, Miles and Rossi\cite{heller} proved that any non-constant solution is of infinite order when $\rho(B)<\rho(A)\leq1/2$. Kumar and Saini\cite{kumarsaini} consider $\lambda(A)<\rho(A)$ and $B(z)$ a transcendental entire function satisfying either $\rho(B)\neq\rho(A)$ or $B(z)$ having Fabry gap and proved that non-trival solutions of \eqref{eq1} are of infinite order. J. Wang and I. Laine\cite{wanglaine} consider $A(z)=h(z)e^{-z}$ and $B(z)$ to satisfy \begin{equation}\label{eq2}
T(r,B)\sim\log M(r,B)
\end{equation}
in a set $E$ satisfying $\overline{\log dens}(E)>0$, where the above notation means $$\lim\limits_{r\to\infty}\frac{T(r,B)}{\log M(r,B)}=1.$$
They have proved the following result.
\\

\textbf{Theorem A}\cite{wanglaine} Suppose that $A(z)$ and $B(z)$ are entire functions where $A(z)=h(z)e^{-z}$ and $\rho(h)<\rho(B)=1$, and that $B(z)$ satisfies \eqref{eq2} in a set of positive upper logarithmic density. Then every non-trivial solution $f$ of equation \eqref{eq1} is of infinite order.\\
\\
They considered $\rho(A)=\rho(B)=1$.  We have considered $\lambda(A)<\rho(A)=n$ and improved above result as follows.

\begin{theorem}
Let $A(z)$ satisfies $\lambda(A)<\rho(A)$ and $B(z)$ satisfies \eqref{eq2} in a set $E$ of positive upper logarithmic density. Then all non-trivial solutions of equation (\ref{eq1}) are of infinite order.
\end{theorem}
For an entire function $\displaystyle{f(z)=\sum\limits_{n=0}^{\infty} a_{\lambda_n}z^{\lambda_n}}$, we say that it has Fe$\acute{j}$er gaps if the infinite series $\displaystyle{\sum\limits_{n=0}^{\infty}\frac{1}{\lambda_n}}$ converges. Murai \cite{murai} proved that function having Fe$\acute{j}$er gaps satisfies \eqref{eq2} in a set of positive upper logarithmic density and has no deficient value. Thus, we can state the corollary with $A(z)$ satisfying condition of previous theorem and $B(z)$ having Fe$\acute{j}$er gap. 

\begin{crlly}Let $A(z)$ satisfies the condition given in Theorem 1 and $B(z)$ has Fejer gaps in a set of positive upper logarithmic density. Then all solutions of \eqref{eq1} are of infinite order.
\end{crlly}
\begin{eg}$A(z)=e^z$ and $\displaystyle{B(z)=\sum\limits_{n=0}^{\infty}\frac{z^{n^2}}{(n^2)!}}$ satisfies the condition of Theorem 1, equation  \eqref{eq1} has all non-trivial solutions of infinite order.
\end{eg}
Kwon\cite{kwon} considered the case of $A(z)$ to be of non-intergral order and proved that all non-trivial solutions of \eqref{eq1} are of infinite order when $0<\rho(B)<1/2$.
\\
 
\textbf{Theorem B}\cite{kwon} Suppose that $A(z)$ is an entire function of finite non-integral order with $\rho(A)>1$, and that all the zeros of $A(z)$ lie in the angular sector  $\theta_1<argz<\theta_2$ satisfying
$$\theta_2-\theta_1<\frac{\pi}{p+1}$$ if $p$ is odd, and
$$\theta_2-\theta_1<\frac{(2p-1)\pi}{2p(p+1)}$$ if $p$ is even, where $p$ is the genus of $A(z)$. Let $B(z)$ be an entire function with
$0<\rho(B)<1/2$. Then every non-constant solution $f$ of \eqref{eq1} has infinite order.
\\

It remains open what will be the behaviour of solution of \eqref{eq1} when $\rho(B)\geq\frac{1}{2}$. For this situation we proposed the following result.
Before stating next result we are giving definition of Fabry gap. For the entire function $\displaystyle{f(z)=\sum\limits_{n=0}^{\infty} a_{\lambda_n}z^{\lambda_n}}$ we say it has Fabry gap if $\displaystyle{\left\{\frac{\lambda_n}{n}\right\}}$ diverges when $n\to\infty$ 
\begin{theorem}
Suppose that $A(z)$ is an entire function of finite non-integral order with $\rho(A)>1$, and that all the zeros of $A(z)$ lie in the angular sector  $\theta_1<argz<\theta_2$ satisfying
$$\theta_2-\theta_1<\frac{\pi}{p+1}$$ if $p$ is odd, and
$$\theta_2-\theta_1<\frac{(2p-1)\pi}{2p(p+1)}$$ if $p$ is even, where $p$ is the genus of $A(z)$. Let $B(z)$ be an entire function with Fabry gap. Then all non-trivial solutions $f$ of equation \eqref{eq1} are of infinite order.
\end{theorem}
In next section we list the preliminary results going to be used to prove the theorem and in section 3 we give proof of proposed theorem.

\section{Preliminary lemmas}
In this section we state some result which will be useful in proving our main results. We denote upper and lower logarithmic density of set $E$ by $\overline{\log dens}(E)$ and $\underline{\log dens}(E)$. Linear measure of a set $E$ is denoted by $m(E)$. Following two lemmas provide estimate of meromorphic function of finite order.
\begin{lemma}\cite{fuchs}\label{l1}
Let $f(z)$ be a meromorphic function of finite order $\rho$. Given $\xi> 0$ and $\delta$, $0 <\delta<1/2$, there is a constant $K(\rho,\xi)$ such that for all $r$ in a set $E$ of lower logarithmic density greater than $1-\xi$ and for every interval $J$ of length $\delta$ $$r\int\limits_J\left|\frac{f'(re^{\iota\theta})}{f(re^{\iota\theta})}\right|d\theta < K(\rho,\xi)(\delta\log\frac{1}{\delta})T(r,f).$$
\end{lemma}
\begin{lemma}\cite{gundersen}\label{l2}
 Let $f$ be a transcendental meromorphic function with
finite order and $(k, j)$ be a finite pair of integers that satisfies $k > j \geq 0$ and let $\epsilon>0$ be a given constant. Then following statements holds:
\begin{enumerate}[(a)]
\item  there exists a set $E_1 \subset [0, 2\pi]$ with linear measure zero such that for $\theta \in [0, 2\pi) \setminus E_1$ there exist $R(\theta) >0$ such that 
$$\left|\frac{f^{(k)}(z)}{f^{(j)}(z)}\right|\leq|z|^{(k-j)(\rho(f)-1+\epsilon)}$$
 for all $k, j$; $| z |> R(\theta)$ and $arg z = \theta$
\item there exists a set $E_2 \subset (1,\infty)$ with finite logarithmic measure  such that for all $|z|\not\in E_2\cup[0,1]$ such that inequality in (a) holds  for all $k, j$ and $| z |\geq R(\theta)$.
\item there exists a set $E_3 \subset [0,\infty)$ with finite linear measure  such that for all $|z|\not\in E_3$ such that 
$$\left|\frac{f^{(k)}(z)}{f^{(j)}(z)}\right|\leq|z|^{(k-j)(\rho(f)+\epsilon)}$$
 holds  for all $k, j$.
\end{enumerate}
\end{lemma}
Following lemma is due to Bank, Laine and Langley\cite{banklainelangley} that gives an estimate for an entire function with integral order.
\begin{lemma}\cite{banklainelangley}\label{l3}
Let $A(z) = h(z)e^{P(z)}$ be an entire function with $\lambda(A) < \rho(A) = n$, where $P(z)$ is a polynomial of degree $n$. Then for
every $\epsilon > 0$ there exists $E \subset [0, 2\pi)$ of linear measure zero satisfying
\begin{enumerate}[(i)]
\item for $\theta \in [0, 2\pi)/setminus E$ with $\delta(P, \theta) > 0$, there exists $R > 1$ such that
$$exp ((1 - \epsilon)\delta(P, \theta)r^n) \leq |A(re^{\iota\theta})|$$
for $r > R$;
\item for $\theta \in [0,2\pi)\setminus E$ with $\delta(P, \theta) < 0$, there exists $R > 1$ such that
$$|A(re^{\iota\theta})| \leq exp ((1 - \epsilon)\delta(P, \theta)r^n)$$ 
for $r > R$.
\end{enumerate}
\end{lemma}
The next lemma gives estimates of function with non-integral order at particular value of $z$. 
\begin{lemma}\cite{kwon}\label{l4}
Let $f(z)$ be an entire function of finite non-integral order $\rho$ and of genus $p>1$. Suppose that for any given $\epsilon>0$, all the zeros of $f(z)$ have their arguments in the following set of real numbers:
$$S(p,\epsilon)=\displaystyle{\{\theta:|\theta|\leq\frac{\pi}{2(p+1)}-\epsilon\}}$$ if $p$ is odd, and
$$S(p,\epsilon)=\displaystyle{\{\theta:\frac{\pi}{2p}+\epsilon\leq|\theta|\leq\frac{3\pi}{2(p+1)}-\epsilon\}}$$ if $p$ is even. Then for any $c>1$, there exists a real number $R>0$ such that $$|f(-r)|\leq\exp(-cr^p)$$
for all $r\geq R$.
\end{lemma}
The following lemma gives the property of an entire function with Fabry gap and can be found in \cite{long} and \cite{wuzheng}.
\begin{lemma}\label{l5}
Let $g(z)=\sum\limits_{n=0}^\infty a_{\lambda_n}z^{\lambda_n}$ be an entire function of finite order with Fabry gap, and $f(z)$ be an entire function with $\rho(f)\in (0,\infty)$. Then for any given $\epsilon \in (0, \rho(f))$, there exists a set $H \subset (1,+\infty)$ satisfying $\overline{logdense}(H) \geq \xi$, where $\xi \in (0, 1)$ is a constant such that for all $|z| = r \in H$, one has
$$\log M(r, h) > r^{\rho(f)-\epsilon}, \log m(r, g) > (1 - \xi)\log M(r, g),$$ where $M(r, h) = max\{ |h(z)| : |z| = r\}$ , $m(r, g) = min\{ |g(z)| :|z| = r\}$ and $M(r, g) = max\{ |g(z)| : |z| = r\}.$
\end{lemma}

 \section{Proof of main theorems}
Before giving the proof of our results we state and prove a lemma which will be used in the proof of Theorem 1.
\begin{lemma}\label{l6}
Let $f(z)$ be an entire function satisfying $T(r,f)\sim\log M(r,f)$ in a set $E$ of positive upper logarithmic density. For given $0<c<\frac{1}{4}$ and $r\in E$, the set $$I_r=\{\theta\in [0,2\pi): \log |f(re^{\iota\theta})|\leq (1-c)\log M(r,f)\}$$ has linear measure zero. 
\end{lemma}
\begin{proof}
Since $$\lim\limits_{r\to\infty}\frac{T(r,f)}{\log M(r,f)}=1$$ in a set $E$ satisfying $\overline{logdens(E)}>0$. Therefore for given $0<\epsilon <c$, we have $$\displaystyle{1-(c-\epsilon)\leq\frac{T(r,f)}{\log M(r,f)}\leq 1+(c-\epsilon)}$$ for all $r>R(\epsilon )$ and $r\in E$. Let $\overline{I_r}$ denotes the complement of $I_r$, then
\begin{align*}
m(r,f)&=\frac{1}{2\pi}\int\limits^{2\pi}_0 \log^+|f(re^{\iota\theta})|d\theta\\
&=\frac{1}{2\pi}\int\limits_{I_r} \log^+|f(re^{\iota\theta})|d\theta+\frac{1}{2\pi}\int\limits_{\overline{I_r}} \log^+|f(re^{\iota\theta})|d\theta\\
&\leq  (1-c)\log^+M(r,f)+\log^+M(r,f)\frac{m(\overline{I_r})}{2\pi}
\end{align*}
Since $f(z)$ is an entire function therefore
$$1-(c-\epsilon)\leq\frac{T(r,f)}{log^+M(r,f)}\leq (1-c)+\frac{m(\overline{I_r})}{2\pi}$$ for all $r>R(\epsilon)$ and $r\in E$. This implies that $$2\pi\epsilon\leq m(\overline{I_r})$$ for all $r>R(\epsilon)$ and $r\in E$. Since $$m(I_r)+m(\overline{I_r})=2\pi$$ this gives $$m(I_r)\leq2\pi(1-\epsilon)\leq\epsilon_1$$ for all $r>R(\epsilon_1)$ and $r\in E$, where $\epsilon_1>0$ and $\epsilon>(1-\frac{\epsilon_1}{2\pi})$.
 \end{proof}
Now we are prepared to give the proof of the theorems.
\\
\begin{center}
\textbf{Proof of Theorem 1:}
\end{center}

\begin{proof}
Let us suppose that solution $f$ of \eqref{eq1} is of finite order. For given $0<c<\frac{1}{4}$, let $$I_r=\{\theta\in [0,2\pi): \log |B(re^{\iota\theta})|\leq (1-c)\log M(r,B)\}.$$ Since, $B(z)$ satisfies $T(r,B)\sim\log M(r,B)$ in a set $E$ satisfying $\overline{\log dens}(E)>0$, using Lemma \ref{l5} we have $m(I_r)\to0$ as $r\to\infty$ and $r\in E$. Using Lemma \ref{l1}, for given $\xi$, $0<\delta<1/2$ and choosing $\delta$ so small that $K(\rho,\xi)(\delta\log\frac{1}{\delta})<c$, we have 
$$r\int\limits_J\left|\frac{B'(re^{\iota\theta})}{B(re^{\iota\theta})}\right|d\theta < cT(r,f)$$
where $K(\rho,\xi)$ is a constant and $r\in E'$ satisfying $\underline{\log dens}(E')>1-\xi$, for every interval $J$ of length $\delta$. If $\theta,\theta'\in [0,2\pi)/I_r$, $|\theta-\theta'|<\delta$ and for all sufficiently large $r\in E\cap E'$, we have
 \begin{align*}
 |B(re^{\iota\theta})|&= \log|B(re^{\iota\theta'})|+\int\limits_{\theta'}^{\theta}\frac{d}{d\theta}\log|B(re^{\iota\theta})|d\theta\\
&> (1-c)\log M(r,B)-r\int\limits_{\theta'}^{\theta}\left|\frac{B'(re^{\iota\theta})}{B(re^{\iota\theta})}\right|d\theta \\
& > (1-2c)\log M(r,B).
 \end{align*}
Using Lemma \ref{l2}(a), we have 
\begin{equation}\label{p11}
\left|\frac{f^{(k)}(re^{\iota \theta})}{f(re^{\iota \theta})}\right| \leq r^{k\rho(f)} 
\end{equation}
for $\theta\in[0,2\pi)/E''$, where $E''$ is a set with linear measure $0$ and $r>R(\theta)$. Choosing $\theta\in[0,2\pi)/E'''$ such that $\delta(P,\theta)<0$, where $E'''$ is a set with linear measure $0$ we have
\begin{equation}\label{p12}
|A(re^{\iota\theta})| \leq exp ((1 - \epsilon)\delta(P, \theta)r^n)
\end{equation} 
for $r > R$.

From equation \eqref{eq1}, \eqref{p11} and \eqref{p12} for $r>R(\theta)$ such that $r\in E\cap E'$, $\theta\in[0,2\pi)/I_r\cup E''\cup E'''$ and $\delta(P,\theta)<0$, we have
$$|B(z)|\leq\left|\frac{f''(re^{\iota\theta})}{f(re^{\iota\theta})}\right|+|A(z)|\left|\frac{f'(re^{\iota\theta})}{f(re^{\iota\theta})}\right|$$
$$M(r,B)^{1-2c}<|B(z)|\leq (1+o(1))r^{2\rho(f)}$$
$$M(r,B)<(1+o(1))r^{4\rho(f)}$$
Thus, $\displaystyle{\frac{\log M(r,B)}{\log r}}$ converges to a finite number as $r\to\infty$, which is a contradiction for transcendental entire function.

\end{proof}

\begin{center}
\textbf{Proof of Theorem 2:}
\end{center}

\begin{proof}
Let us suppose $f$ be a solution of \eqref{eq1} of finite order. Then by Lemma \ref{l2} (c) for $r>R$ and $r\not\in G$ where $G\in[0,\infty)$ is a set of finite linear measure, we have 
\begin{equation}\label{p21}
\left|\frac{f^{(k)}(re^{\iota \theta})}{f(re^{\iota \theta})}\right| \leq r^{k\rho(f)}
\end{equation}
Let us rotate the axes of the complex plane, assume that all the zeros of $A(z)$ have their arguments in the set $S(p,\epsilon)$ defined in Lemma 4 for some $\epsilon>0$. Hence, by Lemma 4, there exists a positive real number $R$ such that for all $r>R$, we have
\begin{equation}\label{p22}
\min_{|z|=r}|A(z)|\leq|A(-r)|\leq\exp(-cr^p)< 1.
\end{equation}

Since $B(z)$ has Fabry gap, by Lemma \ref{l5} there exist a set $H \subset (1,+\infty)$ satisfying $\overline{\log dens}(H) \geq \xi$, where $\xi \in (0, 1)$ is a constant such that for all $|z| = r \in H$. \begin{equation}\label{p23}
(1-\xi)\log M(r,B)<\log m(r,B)<\log |B(z)|
\end{equation}
From equation \eqref{eq1}, we get
$$|B(z)|\leq\left|\frac{f''(re^{\iota\theta})}{f(re^{\iota\theta})}\right|+|A(z)|\left|\frac{f'(re^{\iota\theta})}{f(re^{\iota\theta})}\right|$$
From \eqref{p21}, \eqref{p22} and \eqref{p23}, for $r\in H/G$,  $r>R$ and $\theta\in\{\theta:\min_{|z|=r}|A(z)|=|A(z)|\}$, we have 
$$M(r,B)\leq r^{4\rho(f)}$$ which is a contradiction for transcendental entire function $B(z)$.
\end{proof}

\textbf{Acknowledgement} We are thankful to Dr. Dinesh Kumar for reading the manuscript and suggesting some changes.


\begin{thebibliography}{}
\bibitem{banklainelangley}
S. Bank, I. Laine and J. Langley, On the frequency of zeros of solutions of second order linear differential equation, Results Math., 10 ( 1986), 8-24.
\bibitem{fuchs}
W. Fuchs, Proof of a conjecture of G. Polya concerning gap series, Illinois J. Math., 7 (1963), 661-667.
\bibitem{gundersen}
G. G. Gundersen, Estimates for the Logarithmic Derivative of a Meromorphic Function, J. London Math. Soc., 37,17:1 (1988) 88-104.
\bibitem{gundersen2}
G.G. Gundersen, Finite order solutions of second order linear differential equations, Trans. Amer. Math. Soc. 305 (1988), 415-429.
\bibitem{heller}
S. Hellerstein, J. Miles and J. Rossi, On the growth of solutions of $f'' +gf' +hf = 0$, Trans. Amer. Math. Soc. 324 (1991), 693-706.
\bibitem{kwon}
K. Kwon, Nonexistence of finite order solutions of certain second order linear differential equations, Kodai Math. J. 19 (1996), 378—387.
\bibitem{kwonkim}
K. Kwon and J. Kim, Maximum modulus, characteristic, deficiency and growth of solutions of second order linear differential equations, Kodai Math. J. 24 (2001) 344–351.
\bibitem{kumarsaini}
S. Kumar and M. Saini, On zeros and growth of solutions of second
order linear differential equations, Commun. Korean Math. Soc. 35 (2020), No. 1,  229–241.
\bibitem{lainewu}
I. Laine and P. Wu, Growth of solutions of second order linear differential equations, Proc. Amer. Math. Soc. 128 (2000) 2693–2703.
\bibitem{long}
J. R. Long, Growth of solutions of second order complex linear differential equations with entire coefficients, Filomat., 32 (2018), 275-284.
\bibitem{murai}
T. Murai, The deficiency of entire functions with Fejer gaps, Ann. Inst. Fourier (Grenoble) 33 (1983), 39-58.
\bibitem{wanglaine}
J. Wang and I. Laine, Growth of solutions of second order linear differential equations, J. Math. Anal. Appl. 342(2008), 39-51.
\bibitem{wuzheng}
S. Z. Wu and X. M. Zheng, On meromorphic solutions of some linear differential equations with entire coefficients being Fabry gap series, Adv. Difference Equ., 2015:32 (2015), 13pp.
\end{thebibliography}
\end{document}